\documentclass{amsart}
\usepackage{amsmath,amsfonts,amssymb,amscd,amsthm}
\usepackage{comment}
\usepackage{amsxtra}
\usepackage{color}
\usepackage{amsfonts}
\newtheorem{theorem}[equation]{Theorem}
\newtheorem{lemma}[equation]{Lemma}
\newtheorem{proposition}[equation]{Proposition}
\newtheorem{definition}[equation]{Definition}
\newtheorem{corollary}[equation]{Corollary}

\theoremstyle{remark}
\newtheorem{remark}[equation]{Remark}
\newtheorem{example}[equation]{Example}

\numberwithin{equation}{section}
\numberwithin{equation}{section}
\newcounter{subnumber}
\numberwithin{subnumber}{equation}
\def\Subeq#1#2{\refstepcounter{subnumber}
                 \ifx#1*\relax\else\label{#1}\fi
                 $$#2\leqno{(\thesubnumber})$$}

\DeclareMathOperator{\Ind}{Ind}
\DeclareMathOperator{\Res}{Res}

\DeclareMathOperator{\Irr}{Irr}
\DeclareMathOperator{\Poin}{Poin}
\DeclareMathOperator{\Sym}{Sym}

\DeclareMathOperator{\proj}{proj}

\newcommand{\C}{{\mathbb C}}
\newcommand{\F}{{\mathbb F}}

\newcommand{\N}{{\mathbb N}}
\newcommand{\bbZ}{{\mathbb Z}}

\newcommand{\vp}{{\varphi}}
\newcommand{\vpi}{{\varpi}}
\newcommand{\ol}{\overline}

\newcommand\lr{{\longrightarrow\;}}

\newcommand{\CA}{{\mathcal A}}
\newcommand{\CC}{{\mathcal C}}

\newcommand{\CF}{{\mathcal F}}
\newcommand{\CH}{{\mathcal H}}

\newcommand{\CL}{{\mathcal L}}
\newcommand{\CN}{{\mathcal N}}

\newcommand{\GL}{{\text{GL}}}
\newcommand{\SL}{{\text{SL}}}
\newcommand{\Sp}{{\text{Sp}}}

\newcommand{\inv}{^{-1}}

\newcommand{\lexp}[2]{\kern\scriptspace\vphantom{#2}^{#1}\kern-\scriptspace#2}

\newcommand{\wt}{\widetilde}

\newcommand{\half}{\frac{1}{2}}

\newcommand{\texp}{{^{\text{th}}}}
\newcommand{\bff}{\mathbf f}
\newcommand{\bhh}{\mathbf h}
\newcommand\be{\begin{equation}}
\newcommand\ee{\end{equation}}
\newcommand\ot{\otimes}
\begin{document}
\author{G.I.~Lehrer}
\address{
School of Mathematics and Statistics\\
University of Sydney\\
N.S.W. 2006, Australia\\
Fax: +61 2 9351 4534}
\email{gustav.lehrer@sydney.edu.au}

\title[A factorisation theorem]{A factorisation theorem for the coinvariant algebra of a unitary reflection group}
\begin{abstract} We prove the following theorem. 
Let $G$ be a finite group generated by unitary reflections
in a complex Hermitian space $V=\C^\ell$ and let $G'$ be any reflection subgroup
of $G$. Let $\CH=\CH(G)$ be the space of $G$-harmonic polynomials
on $V$. There is a degree preserving isomorphism
$\mu:\CH(G')\ot\CH(G)^{G'}\overset{\sim}{\lr}\CH(G)$ 
of graded $\CN$-modules, where $\CN:=N_{\GL(V)}(G)\cap N_{\GL(V)}(G')$
and $\CH(G)^{G'}$ is the space of $G'$-fixed points of $\CH(G)$.
This generalises a result of Douglass and Dyer for parabolic subgroups of real reflection groups.
An application is given to counting rational conjugates of reductive groups over $\F_q$.
\end{abstract}
\keywords {Unitary reflection group; coinvariant algebra; reductive group}
\subjclass[2010]{Primary 20F55; Secondary 14G05, 20G40, 51F15}
\maketitle
\section{Background and notation}

Much of the background material in this section may be found in \cite[Ch. 9]{LT}.
Let $G$ be a finite group generated by (pseudo)reflections in a complex
vector space $V$ of dimension $\ell>0$. 
It is well known that if
$S$ denotes the coordinate ring of $V$ (identified with the 
symmetric algebra $S(V^*)$ on the dual $V^*$) then $G$ acts contragrediently on $V^*$,
and hence on $S$, and the ring $S^G$ of
polynomial invariants of $G$ on $S$ is free; if  
$F_1,F_2,\dots,F_\ell$ is a set of homogeneous free generators
of $S^G$, then the degrees $d_i=\deg F_i$ ($i=1,\dots,\ell$) are
determined by $G$, and are called the {\sl invariant degrees} of $G$.

If $\CF$ is the ideal of $S$ generated by the elements of $S^G$ which vanish at $0\in V$,
then $S/\CF:=S_G$ realises the regular representation of $G$.
The space $S_G$ is called the {\sl coinvariant algebra} of $(G,V)$.
Since the ideal $\CF$ is graded, $S_G$ is clearly graded, and for each $i$, the graded component
$(S_G)_i$ of degree $i$ is a $G$ module. By a classical result of
Chevalley \cite[Cor 3.31, p. 52]{LT}, taking products of polynomials
defines an isomorphism of $\C G$-modules
\be\label{eq:tensor}
S\overset\sim\lr S^G\otimes_\C S_G.
\ee

For any $\N$-graded $\C$-vector space $W=W_0\oplus W_1\oplus W_2\oplus\dots$ where the $W_i$ are finite dimensional,
we write 
\be\label{eq:poin}
\Poin_W(t)=\sum_{i=0}^\infty \dim_\C(W_i)t^i\in\C[[t]]
\ee
for its Poincar\'e series. If $\dim(W)<\infty$ then $\Poin_W(t)$ is a polynomial. For properties of these series see \cite[Ch. 4]{LT}.

\subsection{The space of harmonic polynomials}\label{ss:harm} If $\CC$ is any graded $G$-stable complement of $\CF$ in $S$,
then  evidently $\CC\cong S_G$ as graded $G$-module. The space of $G$-harmonic polynomials is a canonical such complement,
which we now describe. Let $\CA:=\CA_G$ be the set of reflecting hyperplanes of $G$ and for each hyperplane $H\in\CA$,
let $L_H\in S$ be a linear form such that $H=\ker(L_H)$ and let $e_H:=|G_H|$, where $G_H$ is the (cyclic) group of reflections
in $H$. Define
\be\label{eq:Pi}
\Pi=\Pi_G:=\prod_{H\in\CA_G}L_H^{e_H-1}\in S.
\ee 

Then \cite[Lemma 9.9]{LT} asserts that $\Pi$ is skew, that is, for $g\in G$, $g\Pi=\det_V(g)\Pi$, and further if $A\in S$ is any
skew polynomial, then 
\be\label{eq:skew}
A=B\Pi, \text{ where }B\in S(V^*)^G.
\ee

Now let $S(V)$ be the symmetric algebra on $V$. For $v\in V$, we have the derivation $D_v$ of $S$ defined by 
\[
D_v(A)(x)=\lim_{t\to 0}\frac{A(x+tv)-A(x)}{t}.
\]
It is explained in \cite[\S9.5]{LT} how the map $v\mapsto D_v$ extends to an algebra homomorphism from $S(V)$ to the algebra of differential
operators on $S$. Thus for $a\in S(V)$ we have $D_a:S\to S$; for homogeneous $a$, this operator has degree $-\deg(a)$. The properties of 
this algebra of operators are summarised in the next statement, whose proof may be found in \cite[Ch. 9]{LT}.

\begin{theorem}\label{thm:pair}
Let $G$, $V$, $S=S(V^*)$, $S(V)$ etc be as above. 
\begin{enumerate}
\item Let $a\in S(V)$, $P\in S(V^*)$ and $g\in\GL(V)$. Then $g(D_a(P))=D_{ga}(gP)$.
\item Define a bilinear pairing $[-,-]:S(V)\times S(V^*)\lr \C$ by $[a,P]:= D_a(P)(0)$ for $a\in S(V)$, $P\in S$. 
This pairing is non-degenerate in both variables.
\item The pairing in (ii) is respected by the action of $\GL(V)$; i.e., for $a\in S(V)$, $P\in S(V^*)$ and $g\in\GL(V)$,
we have $[ga, gP]=[a,P]$.
\end{enumerate}
\end{theorem}

Theorem \ref{thm:pair} shows that $[-,-]$ puts the spaces $S=S(V^*)$ and $S(V)$ in $\GL(V)$-equivariant duality.
Now let $S(V)^G$ be the algebra of invariants of $G$ on $S(V)$. Denote by $\bff\subset S(V)$ the analogue of $\CF$ for $S(V)$.
That is, $\bff$ is the ideal of $S(V)$ generated by the elements of $S(V)^G$ with no constant term.

\begin{definition}\label{def:harm}
Define the space $\CH=\CH(G)$ of $G$-harmonic polynomials in $S$ by $\CH(G):=\bff^\perp$. That is, $P\in\CH$ 
if and only if $[a,P]=0$ for all $a\in\bff$.
\end{definition}

The main facts concerning the space $\CH$ which we shall require are summarised in the next statement.

\begin{theorem}\label{thm:harm}
Maintain the above notation.
\begin{enumerate}
\item We have $P\in\CH$ if and only if $D_a(P)=0$ for all $a\in S(V)^G$. That is, $\CH$ is the space of functions which are 
annihilated by the invariant differential operators.
\item The space $\CH$ coincides with $\{D_a(\Pi)\mid a\in S(V)\}$, where $\Pi$ is as defined in \eqref{eq:Pi}. 
\item The space $\CH$ is a $G$-stable complement of $\CF$ in $S(V^*)$. In particular, we have the $N_{\GL(V)}(G)$
equivariant decomposition
\be\label{eq:sum}
S(V^*)=\CH\oplus \CF.
\ee
\end{enumerate}
\end{theorem}


\section{Reflection subgroups--the main theorem.}

A reflection subgroup of the reflection group $G$ in $V$ is a subgroup of $G$ which is generated by some of the reflections
in $G$. For background concerning such groups see \cite{DyL11, T}. They include the parabolic subgroups of $G$
\cite{St, Le04}, but many other subgroups as well.  Let $G'$ be such a subgroup and write $\CH'=\CH(G')$
for its space of harmonic polynomials $\CF'=\CF(G')$, and so on. For any $G$-module $M$, write $M^{G'}$ for its
subspace of $G'$-fixed elements.
Our objective is to prove the following theorem.

\begin{theorem}\label{thm:main}
Let $G$ be a finite group generated by unitary reflections
in a complex Hermitian space $V=\C^\ell$ and let $G'$ be any reflection subgroup
of $G$. Let $\CH=\CH(G)$ be the space of $G$-harmonic polynomials
on $V$. There is a degree preserving isomorphism
\[
\xi:\CH'\ot\CH^{G'}\overset{\sim}{\lr}\CH
\] 
of graded $\CN$-modules, where $\CN:=N_{\GL(V)}(G)\cap N_{\GL(V)}(G')$.
\end{theorem}

\begin{proof}
We begin with the observation that using \eqref{eq:tensor} applied to the reflection group $G'$ acting on $V$, we have 
the linear isomorphism 
\be\label{eq:1}
\CH'\ot S^{G'}\lr S,
\ee
given by multiplying the polynomials in the tensor factors.

 Next, applying \eqref{eq:sum}, we have the $G$ equivariant decomposition
 \be\label{eq:2}
S=\CH\oplus \CF.
\ee

Since both summands on the right side of \eqref{eq:2} are stable under $G$ and hence {\it a fortiori} under $G'$, 
it follows that we have a graded linear isomorphism
\be\label{eq:3}
S^{G'}\cong \CH^{G'}\oplus \CF^{G'}.
\ee

Substituting \eqref{eq:3} into \eqref{eq:1} we obtain a linear isomorphism
\be\label{eq:4}
\CH'\ot\left( \CH^{G'}\oplus \CF^{G'}   \right)\simeq (\CH'\ot\CH^{G'})\oplus  (\CH'\ot\CF^{G'})\overset{\sim}{\lr} S.
\ee

Now evidently the summand $\CH'\ot\CF^{G'}$ is mapped in the multiplication isomorphism \eqref{eq:4} to a subspace of $\CF$,
since $\CF$ is an ideal of $S$. It follows by restricting the map in \eqref{eq:4} to the summand $\CH'\ot\CH^{G'}$
 that we have a surjective degree preserving linear map
 \be\label{eq:5}
 \xi:\CH'\ot\CH^{G'}\lr \frac{S}{\CF}\simeq S_G.
  \ee 
  
  But since $\CH$ realises the regular representation of $G$, we have $\dim(\CH^{G'})=|G|/|G'|$ while evidently $\dim(\CH')=|G'|$,
  so that the dimension of the left side of \eqref{eq:5} is $|G|=\dim(S_G)(=\dim(\CH))$. It follows that $\xi$ is a graded isomorphism, and since 
  $S/\CF\cong S_G\cong\CH$ as graded $G$-modules, it follows that $\xi$ is a graded linear isomorphism.

It remains only to observe that all homomorphisms above evidently respect the action of $\CN$, and the proof is complete.
\end{proof}

\begin{remark}
One of the key ingredients of the proof is \eqref{eq:tensor}, which asserts that $S$ is free as module over $S^G$.
Notice that taking $G'$-invariants in \eqref{eq:tensor} yields further that $S^{G'}$ is free over $S^G$. This was first
noticed by Dyer \cite{DyP}.
\end{remark}

\begin{remark}
The special case of Theorem \ref{thm:main} where $G$ is a finite Coxeter group and $G'$ is a parabolic subgroup
was treated in \cite[Thm. 2.1]{D94}, where the author acknowledges input from M. Dyer, who pointed out that the result
is connected to the discussion of the cohomology of the flag variety in \cite{BGG}. We believe that our proof 
is significantly simpler and our result is more general than that in {\it op. cit.}.
\end{remark}

\section{Complements}

\subsection{Poincar\'e polynomials} Since Poincar\'e polynomials are multiplicative on tensor products, the following statement is clear.
In the statement we use the convention that for any unitary reflection group $G$, $\Poin_G(t):=\Poin_{S/\CF}(t)=\Poin_{\CH}(t)$
(cf. \eqref{eq:poin}).
\begin{corollary}\label{cor:poin}
With notation as in Theorem \ref{thm:main}, we have
\[
\Poin_{G}(t)=\Poin_{G'}(t)\Poin_{\CH^{G'}}(t).
\]
\end{corollary}

We can be a little more explicit about the second factor in the right side above.
Recall that for any finite dimensional $\C G$-module $M$, the fake degree
\[
f_M^{(G)}(t)=f_M(t):=\sum_{i=0}^N(\CH_i,M)_G\;t^i,
\]
where $(\CH_i,M)_G$ is the intertwining number of $M$ with the degree $i$ graded component $\CH_i$ of $\CH=\CH(G)$.

Denote by $\Irr(G)$ the set of equivalence classes of irreducible $\C G$-modules. For $M\in\Irr(G)$
define integers $m(M)\geq 0$ by 
\[
Ind_{G'}^{G}(1)=\sum_{M\in\Irr(G)}m(M)M.
\]

\begin{proposition}\label{prop:poin}
We have 
\be
\Poin_{\CH^{G'}}(t)=\sum_{M\in\Irr(G)}m(M)f_M(t),
\ee
where $f_M(t)$ is the fake degree of $M$ defined above.
\end{proposition}
\begin{proof}
For any $M\in\Irr(G)$, $\dim(M^{G'})=m(M)$. For $\dim(M^{G'})=(M,1)_{G'}=(M,\Ind_{G'}^G(1))_G$
by Frobenius reciprocity. It follows that $\dim(\CH_i)^{G'}=\sum_{M\in\Irr(G)}m(M)(\CH_i,M)_G$.

Multiplying this last relation by $t^i$ and summing over $i$ and $M$ gives the stated formula.
\end{proof}

We give two examples where $G'$ is a non-parabolic reflection subgroup.

\begin{example}\label{ex:cyc}
Let $G=\mu_e$ be the group of $e\texp$ roots of unity acting on $V=\C$ in the obvious way. Let $G'=\mu_d\subseteq G$
be the subgroup of $d\texp$ roots of unity, where $d$ divides $e$. Then $\Pi=X^{e-1}$, $\Pi'=X^{d-1}$, 
$\CH=\langle 1,X,X^2,\dots,X^{e-1} \rangle$, $\CH'=\langle 1,X,X^2,\dots,X^{d-1} \rangle$
and $\CH^{G'}=\langle 1,X^d,X^{2d},\dots,X^{e-d} \rangle$, so that $\xi:\CH'\ot\CH(G)^{G'}\lr\CH$ is given here 
by simple multiplication of the basis elements.
\end{example}

\begin{example}\label{ex:b2}
Let $G$ be the Weyl group of type $B_2$. Let $V$ have orthonormal basis $\{x,y\}$, with $G'=\langle r_x, r_y\rangle$, 
where $r_x$ is the reflection in $x^\perp$ and similarly for $r_y$. Let $V^*$ have dual basis $X,Y$, so that in the above notation $\vp(x)=X$
and $\vp(y)=Y$.

Further, we have
\[
\Pi=XY(X^2-Y^2),\;\; \Pi'=XY,\text{ and}
\]

Moreover $S(V)^{G'}=\langle x^2,y^2\rangle$ and $S(V)^G=\langle  x^2+y^2, x^2y^2\rangle$.

Using this data, it is straightforward to compute that $\CH$ has a basis
\[
1,X,Y,XY, {X^2-Y^2}, {3X^2Y-Y^3},{X^3-3XY^2}, {X^3Y-XY^3},
\]
and that $\CH'$ has basis $1,X,Y,XY$. Evidently $\CH^{G'}$ has basis $1,X^2-Y^2$

Now consider $\xi:\CH'\ot\CH^{G'}\lr \CH$, given by multiplication of polynomials, followed by projection to $\CH$. 
Clearly $1\ot 1,X\ot 1,Y\ot 1,XY\ot 1\mapsto 1,X,Y,XY$. Now $1\ot(X^2-Y^2)\mapsto \proj_{\CH}(X^2-Y^2)=X^2-Y^2$, where
$\proj_{\CH}:S/\CF\overset{\sim}{\lr}\CH$ is the projection with respect to the decomposition $S=\CH\oplus\CF$.
But $X\ot(X^2-Y^2)\mapsto \proj_{\CH}(X^3-XY^2)$. Since $X^3-XY^2=\half(X^3-3XY^2)+\half X(X^2+Y^2)$, with the 
second summand being in $\CF$, $\xi(X\ot (X^2-Y^2))=\half(X^3-3XY^2)$.
Similarly $\xi(Y\ot (X^2-Y^2))=\half(3X^2Y-Y^3)$ and $\xi(XY\ot(X^2-Y^2))=XY(X^2-Y^2)$.
The group $\CN$ is generated by the interchange of $X$ and $Y$, and the $\CN$-equivariance is easily checked.
\end{example}

\begin{remark}
Our main result implies some constraints on which groups could occur as reflection subgroups. Here is one superficial one.
\begin{corollary}
Let the degrees of $G$ be $d_1,d_2,\dots, d_\ell$ and those of $G'$ be $d_1',d_2',\dots, d_\ell'$. Then
\begin{enumerate}
\item $\prod_{i=1}^\ell(1+t+\dots+t^{d_{i}'-1})$ divides $\prod_{i=1}^\ell(1+t+\dots+t^{d_{i}-1})$.
\item For each $n\in\N$
we have $|\{i\text{ such that }n|d_i'\}|\leq |\{i\text{ such that }n|d_i\}|$.
\end{enumerate}
\end{corollary}
The relation (i) is evident because $\Poin_{G}(t)=\prod_{i=1}^\ell(1+t+\dots+t^{d_{i}-1})$ and similarly for $\Poin_{G'}(t)$,
and (ii) follows using the fact
that $t^k-1=\prod_{n|k}\Phi_n(t)$, where $\Phi_n(t)$ is the $n\texp$ cyclotomic polynomial.
\end{remark}

\section{Duality.}

The main result may be formulated without recourse to the projection $S\lr \CH$ by using the dual reflection structure.
We briefly indicate how this is done.

The key point is the fact that $G$ also acts as a reflection group on $V^*$. We denote the $(G,V^*)$-analogues of $\CH$, $\CF$
and $\Pi$ for $G$ on $V$ by $\bhh$, $\bff$ and $\vpi$ and write $\bhh'$ etc for their $G'$ analogues. 
All the statements in \S 1.1 remain true with $\CH$, $\CF$ and $\Pi$ replaced
by $\bhh$, $\bff$ and $\vpi$ respectively. The following Lemma is an easy consequence of 
the basic facts outlined in \S 1.

\begin{lemma}\label{lem:harm}
\begin{enumerate}
\item There is a linear isomorphism $d:\bhh\lr\CH,$ defined for $h\in\bhh$ by $d(h)=D_h(\Pi)$. If $N:=\deg(\Pi)$, then 
for homogeneous $h\in\bhh$, $\deg(d(h))=N-\deg(h)$.  
\item  There is an isomorphism $e:\bhh(G)^{G'}\lr \CH(G)^{G'}$ defined, for $a\in\bhh^{G'}$, by $e(a)=D_{\vpi' a}(\Pi)$.
\end{enumerate}
\end{lemma}

The maps $d$ and $e$ of Lemma \ref{lem:harm} are evidently $\CN$-equivariant, and our main result may now be formulated
as follows.

\begin{proposition}\label{prop:main}
The isomorphism $\xi:\CH'\ot\CH^{G'}\lr \CH$ of Theorem \ref{thm:main} may be explicitly realised as follows. Let $H=D_h(\Pi')\in\CH'$
and $K=D_{\vpi a}(\Pi)\in\CH^{G'}$ where $h\in\bhh'$ and $a\in\bhh^{G'}$ are uniquely defined as in Lemma \ref{lem:harm}.
Then $\xi(H\ot K)=D_{ah}(\Pi)\in\CH$.

Equivalently, the projection of $HK$ onto $\CH(G)$ is $D_{ah}(\Pi)$.
\end{proposition}

\section{An application to reductive groups.} 

Let $G$ be a connected reductive algebraic group defined over the finite field $\F_q$ of $q$ elements, and let 
$F:G\lr G$ be the corresponding Frobenius endomorphism, as in \cite{Le92}, whose notation we adopt here.
 For any $F$-stable subset $H\subseteq G$ we write 
$H^F$ for the (finite) set of $F$-fixed points of $H$. Let $T_0$ be an $F$-stable maximally $F$-split maximal torus of $G$
and $B\supseteq T_0$ be a Borel subgroup. This data determines the Weyl group $W:=N_G(T_0)/T_0$, together 
with its reflection representation in $V:=Y_0\ot_\bbZ \C$, where $Y_0$ is the cocharacter group of $T_0$ and
its root system $\Phi\subseteq Y_0$ as well as a positive subsystem $\Phi_+\subseteq \Phi$ and its corresponding simple system
$\Pi\subseteq\Phi_+$.

Now take $L$ to be any $F$-stable connected reductive subgroup of $G$ which has maximal rank. 
It is well known that such $L$ are characterised as the connected centralisers of semisimple elements of $G^F$.
They include Levi components of parabolic subgroups. We shall be concerned 
with the set $\CL$ of  $G$-conjugates of $L$ and in particular the set $\CL^F$ of  $F$-stable conjugates of $L$. We may therefore assume,
without loss of generality, that $L\supseteq T_0$. It is always the case that such $L$ is the centraliser of an element of $T_0$, but we shall not
use this fact. Let $\Phi'\subseteq\Phi$ be the root system of $L$ with respect to $T_0$. Then $\Phi'_+:=\Phi'\cap\Phi_+$ is a positive system 
in $\Phi'_+$ and there is a unique corresponding simple system $\Pi'\subseteq\Phi'_+$.  Note that unlike in the case treated in \cite{D94},
it is not generally the case that $\Pi'\subseteq\Pi$. Let $W'$ be the Weyl group of $\Phi'$; this is a reflection subgroup of $W$, but not
necessarily a parabolic subgroup. Let $C:=\{w\in W\mid w\Pi'=\Pi'\}$. Then it is well known that $N_W(W')=C\ltimes W'$.
The methods of \cite{HL80} show that $C$ has a large reflection component, but is not generally a reflection group.
It is also well known that if $N^\circ$ denotes the connected component of the group $N$,
\be\label{eq:n}
\frac{N_G(L)}{N_G(L)^\circ}\simeq C,
\ee
where $C$ is as above. Now the conjugate $^gL:=gLg\inv$ is $F$-stable precisely when $g\inv F(g)\in N_G(L)$. 
We therefore have a map $\CL^F\lr C$ obtained by taking the image in $C$ of $g\inv F(g)$ in $C$ (cf. \eqref{eq:n}).
It is easily checked that this image is uniquely determined up to $F$-conjugacy, where the $F$-conjugate of $c$ by $x$
is $xcF(x)\inv$. We denote this image by $\omega(^gL)$.

The following facts are standard and may be found, e.g. in \cite{SpS}, \cite{StE}, \cite{Le92} or  \cite{D94}. 
\begin{lemma}\label{lem:cent} Maintain the above notation.
\begin{enumerate}
\item The map $\omega:\CL^F\lr$ $\{F$-conjugacy classes of $C\}$ described above induces a bijection from the set
of $G^F$-orbits on $\CL^F$ to the set of $F$-conjugacy classes of $C$.
\item Let $^gL\in\CL^F$ and let $Z\subseteq C$ be the $F$-centraliser of $\omega(^gL)$. Then $N_{G^F}(^gL)$ is the semidirect product
of $(^gN^\circ)^F$ with $Z$.
\end{enumerate}
\end{lemma}

Next, recall that $F$ acts on $V$ as $F=qF_0$, where $F_0\in\GL(V)$ fixes both $\Pi$ and $\Pi'$ setwise, since
both $L$ and $G$ are $F$-stable. If $\ol C$ is the group generated by $C$ and $F_0$, then $\ol C\subseteq N_{\GL(V)}(W)\cap N_{\GL(V)}(W')$,
and hence we have an action of $\ol C$ on the space $\CH(W)^{W'}$ of Theorem \ref{thm:main}. We note also that
the conjugacy class in $\ol C$ of $cF_0\inv$ is of the form $\wt c F_0\inv$, where $\wt c$ is an $F$-conjugacy class of $C$.

Let $g$ be a function on $C$ which is constant on $F$-conjugacy classes. Define the corresponding class function $\gamma_g$ on $\ol C$
by $\gamma_g(cF_0^i)=\begin{cases}
g(c)\text{ if }i=-1\\
0\text{ otherwise}\\
\end{cases}$.

In view of Lemma \ref{lem:cent} (i), it is clear that  $g$ also defines a function on $\CL^F$, which is constant on $G^F$-conjugacy classes.
This function is also denoted by $g$.

With this notation we may state the following result.
\begin{theorem}\label{thm:gen} Let $L$ be any $F$-stable connected reductive subgroup of maximal rank of $G$ and let
$W,W',C$ etc. be as above.

Let $g$ be a function on $C$ which is constant on $F$-conjugacy classes. Then
\[
\sum_{L'\in\CL^G}g(L')=q^{2(N-N')}\sum_{d=0}^{N-N'}\langle \CH(W)^{W'}_d, \gamma_g  \rangle_{\ol C}q^{-d},
\]
where $\CH(W)^{W'}_d$ is the $d^{\text {th}}$ graded component of $\CH(W)^{W'}$, $\langle-,-\rangle_{\ol C}$
denotes the usual inner product of class functions on $\ol C$, $N=|\Phi_+|$ and $N'=|\Phi'_+|$.
\end{theorem}

The proof of this result is exactly as in \cite{Le92}, where the result is proved for $L$ equal to a torus,
 and \cite{D94} where the result is proved for $L$ a Levi factor. The crucial difference is that our Theorem \ref{thm:main}
was not available for arbitrary reflection subgroups of $W$ in \cite{D94}. We remark finally that the usual corollaries concerning 
the number of $F$-stable conjugates of $L$ are now  available in the wider generality of our result. As an example, we have
the following result.

\begin{corollary}
Maintaining the above notation, the number of $F$-stable conjugates of $L$ is equal to 
\[
q^{2(N-N')}|F_0|\sum_{d=0}^{N-N'}\langle \CH(W)^{W'}_d, \chi_{CF_0\inv}  \rangle_{\ol C}q^{-d},
\]
where $\chi_{CF_0\inv}$ is the characteristic function of the coset $CF_0\inv$ in $\ol C$. 
\end{corollary}

In particular, if $G$ is split, $\ol C=C$ and this number is just the Poincar\'e polynomial 
\be\label{eq:split}
q^{2(N-N')}\sum_{d=0}^{N-N'}\langle \CH(W)^{W'}_d, 1_C  \rangle_{C}q^{-d}=q^{2(N-N')}\sum_{d=0}^{N-N'}\dim \CH(W)^{W'C}_dq^{-d}.
\ee

We close with two examples where $L$ is not a Levi factor.

\begin{example}\label{ex:sp4} Take $G=\Sp_4(\ol\F_q)$ and let $L$ be the reductive subgroup of maximal rank
corresponding to the unique subsystem of the root system $\Phi$ of type $A_1\times A_1$. Thus $L$ has semisimple
part of type $\SL_2\times\SL_2$. One sees easily that in this case $W'$ is normal in $W$, so that $W'C=W$. 
Further, $|N|=4$ and $|N'|=2$. We may therefore apply \eqref{eq:split} to conclude that the number of $F$-stable
conjugates of $L$ is $q^4$.
\end{example}

\begin{example}\label{ex:sp6} Take $G=\Sp_6(\ol \F_q)$ and let $L$ be a reductive subgroup of maximal rank with
semisimple part isomorphic to $\SL_2\times \Sp_4$. Thus the corresponding root subsystem is of type
$A_1\times C_2$. In this case we have $|\Phi_+|=N=9$, $|\Phi'_+|=N'=5$ and $|C|=2$, so that
$N_W(W')=W'C\simeq \Sym_2\ltimes (\frac{\bbZ}{2\bbZ})^3\subset W\simeq \Sym_3\ltimes (\frac{\bbZ}{2\bbZ})^3$.
To compute $\dim \CH(W)^{W'C}_d$, observe that for any $W$-module $M$ and subgroup $W_1\subseteq W$, we have
$\dim M^{W_1}=\langle \Res^W_{W_1}(M), 1\rangle_{W_1}=\langle M,\Ind_{W_1}^W(1)\rangle_W$ by Frobenius 
reciprocity.

In our case, it is easily verified that $\Ind_{W'C}^W(1)=1+\rho$, where $\rho$ is the two dimensional representation 
of $W$ obtained by pulling back the two dimensional irreducible representation of $\Sym_3$ via the map
$W=\Sym_3\ltimes (\frac{\bbZ}{2\bbZ})^3\lr \Sym_3$. Further it is well known and easily verified that
\[
\begin{aligned}
\langle \CH(W)_d, 1\rangle_W=&
\begin{cases}
0\text{ if }d\neq 0\\
1\text{ if } d=0, \text{ and }\\
\end{cases}\\
\langle \CH(W)_d, \rho\rangle_W=&
\begin{cases}
0\text{ if }d\neq 2\text{ or }4\\
1\text{ if } d=2\text{ or }4.\\
\end{cases}\\
\end{aligned}
\]
We may therefore apply \eqref{eq:split} to deduce that the number of $F$-stable conjugates of $L$ is equal to
$q^4\sum_{d=0}^4\left(\langle \CH(W)_d, 1\rangle_W+\sum_{d=0}^4\langle \CH(W)_d, \rho\rangle_W\right)q^{-d}
=q^4(1+q^{-2}+q^{-4})=1+q^2+q^4$.
\end{example}



\end{document}